\documentclass[reqno,twoside, 11pt]{amsart}
\usepackage{amsfonts,xcolor,eucal}
\usepackage{a4wide,amsmath,amssymb,latexsym,amsthm}
\usepackage[dvips,bottom=1.4in,right=1in,top=1in, left=1in]{geometry}
\usepackage{graphicx,epstopdf,xspace,enumerate,pstricks}
\usepackage{mathrsfs,mathtools,epic,commath,bm}

\usepackage[colorlinks=true,urlcolor=blue,
citecolor=red,linkcolor=blue,linktocpage,pdfpagelabels,
bookmarksnumbered,bookmarksopen]{hyperref}
\usepackage[english]{babel}

\newfont{\script}{eusm10 scaled\magstep1}

\newcommand{\R}{\mathbb{R}}

\numberwithin{equation}{section}

\providecommand{\Real}{\mathop{\rm Re}\nolimits}%
\usepackage{xcolor}
\usepackage{soul}
\definecolor{jddcol}{rgb}{0,0,0.8}
\definecolor{afcol}{rgb}{1,0,0}

\newtheorem{theorem}{Theorem}[section]
\newtheorem{proposition}[theorem]{Proposition}
\newtheorem{lemma}[theorem]{Lemma}

\theoremstyle{definition}
\newtheorem{definition}[theorem]{Definition}
\newtheorem{example}[theorem]{Example}

\theoremstyle{remark}
\newtheorem{remark}{Remark}

\title[General class of fractional-calculus operators]{On fractional calculus with analytic kernels with respect to functions}

\author[Oumarou]{Christian Maxime Steve Oumarou}
\author[Fahad]{Hafiz Muhammad Fahad}
\author[Djida]{Jean-Daniel Djida}
\author[Fernandez]{Arran Fernandez}

\address[Oumarou,Djida]{\small African Institute for Mathematical Sciences (AIMS), P.O. Box 608, Limbe Crystal Gardens,South West Region, Cameroon.}
\email[Oumarou]{christian.oumarou@aims-cameroon.org}
\email[Djida]{jeandaniel.djida@aims-cameroon.org}

\address[Fahad,Fernandez]{\small Department of Mathematics, Eastern Mediterranean University, Famagusta, Northern Cyprus, via Mersin-10, Turkey.}
\email[Fahad]{hafizmuhammadfahad13@gmail.com}
\email[Fernandez]{arran.fernandez@emu.edu.tr}

\thanks{The third author is supported by the Deutscher Akademischer Austausch Dienst/German Academic Exchange Service (DAAD)}



\begin{document}

\begin{abstract}
Many different types of fractional calculus have been proposed, which can be organised into some general classes of operators. For a unified mathematical theory, results should be proved in the most general possible setting. Two important classes of fractional-calculus operators are the fractional integrals and derivatives with respect to functions (dating back to the 1970s) and those with general analytic kernels (introduced in 2019). To cover both of these settings in a single study, we can consider fractional integrals and derivatives with analytic kernels with respect to functions, which have never been studied in detail before. Here we establish the basic properties of these general operators, including series formulae, composition relations, function spaces, and Laplace transforms. The tools of convergent series, from fractional calculus with analytic kernels, and of operational calculus, from fractional calculus with respect to functions, are essential ingredients in the analysis of the general class that covers both.
\end{abstract}

\maketitle

\section{Introduction}
Fractional calculus was introduced around the end of the seventeenth century, as a branch of mathematical analysis that deals with the studies of various possibilities of defining real number powers or complex number powers of the differentiation operator and the integration operator. The expression ``power'' refers to iterative application of a linear operator to a function: repeatedly composing the operator $n$ times gives the $n$th power of the operator, when $n$ is an integer. Fractional calculus is the generalisation of classical calculus concerned with operations of integration and differentiation of non-integer order.

Gottfried Wilhelm Leibniz, one of the inventors of calculus, established the notation $\frac{\mathrm{d}^n}{\mathrm{d}x^n}f(x)$ for the $n$th derivative of a function $f$, for natural numbers $n$. In a letter of 1695, Fran\c cois Antoine de L'H\^{o}pital wrote to Leibniz asking what would happen if $n=\frac{1}{2}$, and the response of Leibniz to L'H\^{o}pital was: \emph{``A paradox, from which it appears that one day useful consequences will be drawn''}. That was the birth of what we call fractional calculus, which has attracted great mathematicians such as Euler, Liouville, Laplace, Riemann, and many others \cite{MillerRoss,Samko}. Nowadays many scientists are interested in the field due to its applications in engineering, physics, chemistry, biology, economics, and so on. These applications arise because fractional calculus is very useful for modelling different types of physical systems: often due to its nonlocality properties, as opposed to classical derivatives which are local, and also the ability of fractional derivatives to capture intermediate behaviours, such as viscoelastic substances which are intermediate between solid and liquid \cite{BDST,Hilfer,KST}.

Aside from the applications of fractional calculus, many advances have also been made in this field from the viewpoint of pure mathematics. Since the study of differential equations is a major field of mathematics, many methods used for solving them have been investigated and extended to a fractional context. Nowadays we have analytical solution methods \cite{kexue-jigen,luchko,restrepo-ruzhansky-suragan}, numerical solution methods \cite{brzezinski,li-chen,ozdemir-yavuz}, qualitative properties of solutions \cite{djida-nieto-area,fernandez,restrepo-suragan}, as well as several extensions of the classical notion of fractional differential equations, such as stochastic \cite{ahmadova-mahmudov}, delay \cite{huseynov-mahmudov}, and variable-order \cite{sun-etal} fractional differential equations.

There are many different ways of defining fractional operators, unlike in classical calculus where there is only one way to define the derivative operation. The Riemann--Liouville fractional calculus is the most commonly used definition, but it is limited by the fact that when used for modelling physical problems, the initial value conditions are fractional, which seems not to be appropriate for physical conditions. In this sense, the Caputo fractional derivative is more preferable since then the initial condition required is in the classical form~\cite{Diethelm}. Among other definitions of fractional calculus, we can mention Hilfer, Riesz, Hadamard, Atangana--Baleanu, Prabhakar, tempered, and many others~\cite{Samko,BF2019}. These definitions are not equivalent to each other, but each one has its own properties and its own applications, and all of them may be useful in different settings: for example, for modelling systems with power law behaviour, logarithmic behaviour, exponential behaviour, and other more complex behaviours, or for different types of initial conditions.

For applications, then, any one individual type of fractional calculus may be useful. But from the mathematical point of view, it is not efficient to prove the same properties and theorems over and over again for every fractional operator. Mathematicians seek to generalise, to prove results just once in a general setting and then apply them for each operator as special cases. Therefore, it is proposed to study broad general classes of operators instead of analysing specific operators individually.

The class of fractional integrals and derivatives with respect to functions is an important subtopic of fractional calculus, dating back to Erdelyi and Osler in the 1960s and 1970s~\cite{Almeida2017,Erdelyi,Osler}. The class of fractional integrals and derivatives with analytic kernels is a newer innovation, proposed in 2019~\cite{BF2019,FOB}. Each of these two classes is broad enough to cover various specific types of fractional calculus and capture diverse behaviours in fractional systems. Combining both ideas yields a third, even more general, class: that of fractional integrals and derivatives with analytic kernels with respect to functions. This covers both of the two above-mentioned classes as subclasses; it was briefly defined in~\cite{FOB}, but its properties have never been studied in the published literature so far. In this work, we undertake the first formal study of fractional calculus with analytic kernels with respect to functions.

The organisation of this paper is as follows. In Section~\ref{chap:2}, we review fractional calculus: firstly in the classical Riemann--Liouville and Caputo definitions, secondly the class of fractional calculus with respect to functions, thirdly the class of fractional calculus with analytic kernels. In section~\ref{chap:5}, we consider fractional calculus with analytic kernels with respect to functions, and prove various new results concerning these generalised operators, such as establishing series formulae and composition properties. In section~\ref{chap:function}, we consider functional analysis of these new operators, establishing appropriate function spaces in which they can be applied. In section~\ref{chap:Lap}, we show how some fractional integro-differential equations using the new operators may be solved using a type of generalised Laplace transform. Finally, in section~\ref{chap:concl} we conclude the paper and look towards plans for future research.

\section{Preliminaries}\label{chap:2}

In this section, we provide the definitions and fundamental properties of the fractional integral and derivative operators which will be used and referred to throughout the work. We shall start by introducing the well-known operators of Riemann--Liouville and Caputo, on which much of fractional calculus is based, and then continue with the two firmly established general classes mentioned above: fractional calculus with respect to functions, and fractional calculus with general analytic kernels.

\subsection{Riemann--Liouville and Caputo fractional calculus}

\begin{definition}[Fractional integral] \label{Def:RLi}
	Let $a\in\mathbb{R}$ be a constant, $u\in L^1[a,b]$ a function, and $\alpha\in \mathbb{C}$ such that $\Real(\alpha)>0$. Then, the Riemann--Liouville fractional integral of the function $u$ to order $\alpha$ with respect to $x$ and with constant of integration $a$ is defined by:
	\begin{align*}
	\prescript{RL}{a}I^{\alpha}_{x} u(x)\coloneqq \frac{1}{\Gamma(\alpha)} \int_{a}^{x} (x-t)^{\alpha-1} u(t) \,\mathrm{d}t. 
	\end{align*}
\end{definition}

\begin{definition}[Riemann--Liouville fractional derivative] \label{Def:RLd}
	Let $a\in\mathbb{R}$ be a constant, $u\in C^n[a,b]$ a function, and $\alpha\in \mathbb{C}$ such that $\Real(\alpha)\geq0$. Then, the Riemann--Liouville fractional derivative with order $\alpha$ of the function $u$ with respect to $x$ and with constant of integration $a$ is given by: 
	\begin{align*}
	\prescript{RL}{a}D^{\alpha}_{x} u(x)\coloneqq \frac{\mathrm{d}^n}{\mathrm{d}x^n}\left(\prescript{RL}{a}I^{n-\alpha}_{x}u(x)\right),\;\; n=\lfloor\Real(\alpha)\rfloor +1, 
	\end{align*}
	where we note that $n$ is chosen so that $n\in \mathbb{N}$ and $\Real(n-\alpha)>0$, to give a classical derivative and a fractional integral.
\end{definition}

\begin{remark} By convention, we have 
\[
\prescript{RL}{a}D^{-\alpha}_{x} u(x)=\prescript{RL}{a}I^{\alpha}_{x} u(x),
\]
so that the fractional operators $\prescript{RL}{a}D^{\alpha}_{x}$ and $\prescript{RL}{a}I^{\alpha}_{x}$ are defined for all $\alpha\in\mathbb{C}$. We use the word ``differintegral'' for an operator which may be either a fractional derivative or a fractional integral.
\end{remark}

\begin{remark} When the order of differintegration is a complex number, the difference between fractional derivative and fractional integral is at the level of the sign of the real part. If $\Real(\alpha)=0$, pure imaginary order, then we have to use Definition \ref{Def:RLd} (fractional derivative) only because Definition \ref{Def:RLi} (fractional integral) requires $\Real(\alpha)>0$. This implies
\begin{align*}
\prescript{RL}{a}D^{i\theta}_{x} u(x)\coloneqq \frac{\mathrm{d}}{\mathrm{d}x} \left(\prescript{RL}{a}I^{1-i\theta}_{x} u(x)\right).
\end{align*}
\end{remark}

\begin{definition}[Caputo fractional derivative]
	Let $a\in\mathbb{R}$ be a constant, $u\in C^n[a,b]$ a function, and $\alpha\in \mathbb{C}$ such that $\Real(\alpha)\geq0$. Then, the Caputo fractional derivative of the function $u$ with order $\alpha$ with respect to $x$ with constant of integration $a$ is defined by:
	\begin{align*}
	\prescript{C}{a}D^{\alpha}_{x} u(x)&\coloneqq \prescript{RL}{a}I^{n-\alpha}_{x}\left(\frac{\mathrm{d}^n}{\mathrm{d}x^n} u(x)\right),\;\; n\coloneqq \lfloor\Real(\alpha)\rfloor+1,\;\; \Real(\alpha)>0.	
	\end{align*}
In Caputo fractional calculus, fractional integrals are defined by the formula which gives the Riemann--Liouville integral.
\end{definition}

Following the above definitions, we state some results on the compositions of fractional operators.  

\begin{lemma}\label{lem:sg} 
Riemann--Liouville fractional integrals are commutative and have a semigroup property. That is,
	\begin{align*}
	\prescript{RL}{a}I^{\alpha}_{x}\left(\prescript{RL}{a}I^{\beta}_{x} u(x)\right)=\prescript{RL}{a}I^{\beta}_{x}\left(\prescript{RL}{a}I^{\alpha}_{x} u(x)\right)=\prescript{RL}{a}I^{\alpha+\beta}_{x} u(x),
	\end{align*}  
for any $\alpha,\beta\in\mathbb{C}$ such that $\Real(\alpha)>0$, $\Real(\beta)>0$ and any $a\in\mathbb{R}$, $u\in L^1[a,b]$.
\end{lemma}

\begin{lemma}\label{lem:st} 
Standard ($\mathbb{N}$-order) derivatives of Riemann--Liouville differintegrals have a semigroup property. That is, 
	\begin{align*}
	\prescript{RL}{a}D^{n}_{x}\left(\prescript{RL}{a}D^{\alpha}_{x} u(x)\right)&=\prescript{RL}{a}D^{n+\alpha}_{x} u(x),
	\end{align*}
for any $\alpha \in \mathbb{C}$, $n\in \mathbb{N}$, $a\in \mathbb{R}$, and sufficiently differentiable function $u$ on $[a,b]$.
\end{lemma}

\begin{lemma}\label{lem:RLS} 
Riemann--Liouville fractional differintegrals of Riemann-Liouville fractional integrals have a semigroup property. That is,
	\begin{align*}
	\prescript{RL}{a}D^{\alpha}_{x} \left(\prescript{RL}{a}D^{\beta}_{x} u(x)\right)=\prescript{RL}{a}D^{\alpha+\beta}_{x} u(x),
	\end{align*}
for any $\alpha,\beta\in \mathbb{C}$ such that $\Real(\beta)<0$, any $a\in\mathbb{R}$ and sufficiently differentiable function $u$ on $[a,b]$.
\end{lemma}

\subsection{Fractional calculus with respect to functions} \label{Subsec:wrt}
Here, we review the fractional operators seen in the previous part but now with respect to functions~\cite{Almeida2017,Erdelyi,Osler,FFS}.

\begin{definition}[Fractional integral with respect to function] \label{Def:Iwrt}
	Let $a\in\mathbb{R}$ be a constant, $u\in L^1[a,b]$ a function, $\alpha\in \mathbb{C}$ such that $\Real(\alpha)>0$, and $\varphi$ a monotonic $C^1[a,b]$ function. Then, the Riemann--Liouville fractional integral of the function $u$ to order $\alpha$ with respect to the function $\varphi$ and with constant of integration $a$ is defined by:
	\begin{align*}
	\prescript{}{a}I^{\alpha}_{\varphi(x)}  u(x)&=\frac{1}{\Gamma(\alpha)} \int_{a}^{x} \varphi'(t)[\varphi(x)-\varphi(t)]^{\alpha-1} u(t)\,\mathrm{d}t,\label{99}
	\end{align*} 
	which is the fractional power of the operator of differentiation $\frac{\mathrm{d}}{\mathrm{d}\varphi(x)}=\frac{1}{\varphi'(x)}\cdot\frac{\mathrm{d}}{\mathrm{d}x}$.
\end{definition}

\begin{example} For $\varphi(x)=x$, we recover the Riemann--Liouville fractional integral of order $\alpha$ and with respect to $x$.
\end{example}

\begin{example} For $\varphi(x)=\log x$, we have
\begin{align*}
\prescript{}{a}I^{\alpha}_{\log x}u(x)=\frac{1}{\Gamma(\alpha)} \int_{a}^{x} \left(\log\frac{x}{t}\right)^{\gamma-1}\frac{u(t)}{t}\,\mathrm{d}t= \prescript{H}{a}I^{\alpha}_{x}u(x),
\end{align*}
the Hadamard fractional integral of order $\alpha$.
\end{example}

\begin{definition}[Riemann--Liouville fractional derivative with respect to function] \label{Def:Dwrt}
Let $a, b\in\mathbb{R}$ be two constants such that $a < b$, and $u,\varphi\in C^{n}[a,b]$ be two functions such that $\varphi'(x)>0$ for all $x\in[a,b]$, and $\alpha\in \mathbb{C}$ be such that $\Real(\alpha)\geq0$. Then, the Riemann--Liouville fractional derivative with order $\alpha$ of the function $u$ with respect to the function $\varphi$ and with constant of integration $a$ is given by: 
	\begin{align*}
	\prescript{RL}{a}D^{\alpha}_{\varphi(x)}&u(x)= \left(\frac{1}{\varphi'(x)}\cdot\frac{\mathrm{d}}{\mathrm{d}x}\right)^n\prescript{RL}{a}I^{n-\alpha}_{\varphi(x)}u(x),
	\end{align*}
	with $n-1\leq \Real(\alpha) < n\in \mathbb{Z}^{+}$. 
\end{definition} 

\begin{example} For $\varphi(x)=x$, we recover the Riemann--Liouville fractional derivative with respect to $x$.
\end{example}

\begin{definition}[Caputo fractional derivative with respect to function] \label{Def:Cwrt}
Let $a, b\in\mathbb{R}$ be two constants such that $a < b$, and $u,\varphi\in C^{n}[a,b]$ be two functions such that $\varphi'(x)>0$ for all $x\in[a,b]$, and $\alpha\in \mathbb{C}$ be such that $\Real(\alpha)\geq0$. Then, the Caputo fractional derivative of the function $u$ with order $\alpha$ with respect to the function $\varphi$ with constant of integration $a$ is defined by:
	\begin{align*}
	\prescript{C}{a}D^{\alpha}_{\varphi(x)}&u(x)=\prescript{}{a}I^{n-\alpha}_{\varphi(x)} \left(\frac{1}{\varphi'(x)}\cdot\frac{\mathrm{d}}{\mathrm{d}x}\right)^n\;  u(x),
	\end{align*} 
	with $n-1\leq \Real(\alpha)< n\in \mathbb{Z}^{+}$.
\end{definition}



\textbf{Conjugation relation.} The fractional operators with respect to functions can be written as conjugations (in the group theory sense) of the original fractional operators with some operators of composition. This allows many results on fractional operators with respect to functions to be proved in a straightforward way, simply reducing the problem to one that has been already solved for the original Riemann--Liouville and Caputo operators.

Let us define the composition operator $Q_{\varphi}$ by $Q_{\varphi}f=f\circ\varphi$, i.e.,:
\begin{align}
\big(Q_{\varphi} f\big)(x)=f(\varphi(x)).
\end{align}
Note that the inverse operator is given by $Q_{\varphi}^{-1}=Q_{\varphi^{-1}}$, for any monotonic function $\varphi$. Then, the following proposition gives us some conjugation relations.

\begin{proposition}\label{pro:x1}
With all notation as in Definition \ref{Def:Iwrt}, we have the following property:
	\begin{align*}
	\prescript{RL}{a}I^{\alpha}_{\varphi(x)} = Q_{\varphi}\;\circ\;\prescript{RL}{\varphi{(a)}}I^{\alpha}_{x}\;\circ\;Q_{\varphi}^{-1}.
	\end{align*} 
With all notation as in Definition \ref{Def:Dwrt}, we have the following property:
	\begin{align*}
	\prescript{RL}{a}D^{\alpha}_{\varphi(x)} = Q_{\varphi}\;\circ \prescript{RL}{\varphi(a)}D^{\alpha}_{x}  \circ\;Q_{\varphi}^{-1}.
	\end{align*}
With all notation as in Definition \ref{Def:Cwrt}, we have the following property:
	\begin{align*}
	\prescript{C}{a}D^{\alpha}_{\varphi(x)} = Q_{\varphi}\;\circ \prescript{C}{\varphi(a)}D^{\alpha}_{x}  \circ\;Q_{\varphi}^{-1}.
	\end{align*}
\end{proposition}

In view of the conjugation relations given by Proposition \ref{pro:x1}, the composition properties of fractional integrals and derivatives with respect to functions follow immediately from Lemmas \ref{lem:sg}, \ref{lem:st}, \ref{lem:RLS} on the original Riemann--Liouville differintegrals.

\begin{proposition}\label{pro:a}
Fractional integrals with respect to a function are commutative and have a semigroup property. That is,
	\begin{align*}
	\prescript{RL}{a}I^{\alpha}_{\varphi(x)}\left(\prescript{RL}{a}I^{\beta}_{\varphi(x)} u(x)\right)=\prescript{RL}{a}I^{\beta}_{\varphi(x)}\left(\prescript{RL}{a}I^{\alpha}_{\varphi(x)} u(x)\right)=\prescript{RL}{a}I^{\alpha+\beta}_{\varphi(x)} u(x),
	\end{align*}  
for any $\alpha,\beta\in\mathbb{C}$ such that $\Real(\alpha)>0$, $\Real(\beta)>0$, and any $a\in\mathbb{R}$, $u\in L^1[a,b]$, and any monotonic $C^1[a,b]$ function $\varphi$.
\end{proposition}

\begin{proposition}\label{pro:b}
Standard ($\mathbb{N}$-order) derivatives with respect to a function of Riemann--Liouville differintegrals with respect to the same function have a semigroup property. That is, 
	\begin{align*}
	\left(\frac{1}{\varphi'(x)}\cdot\frac{\mathrm{d}}{\mathrm{d}x}\right)^n\left(\prescript{RL}{a}D^{\alpha}_{\varphi(x)} u(x)\right)&=\prescript{RL}{a}D^{n+\alpha}_{\varphi(x)} u(x).
	\end{align*}
for any $\alpha \in \mathbb{C}$, $n\in \mathbb{N}$, $a\in \mathbb{R}$, sufficiently differentiable function $u$, and monotonic smooth function $\varphi$.
\end{proposition}

\begin{proposition}\label{pro:c}
Differintegrals of fractional integrals, with respect to functions and with derivatives in the Riemann--Liouville sense, have a semigroup property. That is,
	\begin{align*}
	\prescript{RL}{a}D^{\alpha}_{\varphi(x)} \left(\prescript{RL}{a}D^{\beta}_{\varphi(x)} u(x)\right)=\prescript{RL}{a}D^{\alpha+\beta}_{\varphi(x)} u(x),
	\end{align*}
for any $\alpha,\beta\in \mathbb{C}$ such that $\Real(\beta)<0$, any $a\in\mathbb{R}$, and any sufficiently differentiable function $u$ and monotonic smooth function $\varphi$.
\end{proposition}

\subsection{Fractional calculus with analytic kernels}\label{chap:4}

Here, we review the fractional operators with general analytic kernels~\cite{BF2019,FOB}.

\begin{definition}[\cite{FOB}]\label{def:F}
	Let $[a,b]$ be a real interval, $\alpha,\beta\in\mathbb{C}$ with $\Real(\alpha)>0$, $\Real(\beta)>0$, and $R\in\mathbb{R}^+$ satisfying $R > (b-a)^{\Real(\beta)}$. Let $\mathcal{A}$ be a function analytic on the complex disc $\mathcal{D}(0,R)$ and defined on this disc by the locally uniformly convergent power series \begin{align}
	\mathcal{A}(x)=\sum_{n=0}^{\infty} a_n x^n, \label{41}	
	\end{align}
	with the coefficients $a_n=a_n (\alpha, \beta)$ which can depend on $\alpha$ and $\beta$ if desired. Then, the generalised fractional integral of a function $u\in L^1[a,b]$ with kernel $\mathcal{A}$ is given by:
	\begin{align}
\prescript{\mathcal{A}}{a}I^{\alpha, \beta}_{x}u(x) \coloneqq\int_{a}^{x} (x-t)^{\alpha-1} \mathcal{A}\left((x-t)^{\beta}\right)u(t)\, \mathrm{d}t.\label{E}
	\end{align}  
\end{definition}

\begin{remark} \label{def:2}
If $\mathcal{A}$ is an analytic function as defined by \eqref{41}, we define a transformed function $\mathcal{A}_{\Gamma}$ which will be very useful for our work as follows:
	\begin{align*}
	\mathcal{A}_{\Gamma}(x)=\sum_{n=0}^{\infty} a_n  \Gamma(\beta n+\alpha)\; x^n.
	\end{align*}  	
\end{remark}

Under the given assumptions that $\Real(\alpha)>0$ and $\Real(\beta)>0$, the operators $\prescript{\mathcal{A}}{a}I^{\alpha, \beta}_{x}$ defined in Definition \ref{def:F} comprise a commutative family of bounded operators from $L^1[a,b]$ to itself \cite{FOB}.

\begin{proposition}[Series formula \cite{FOB}] \label{lem:max}
With all notation as in Definition~\ref{def:F}, the formula of $\prescript{\mathcal{A}}{a}I^{\gamma, \delta} u$ as a locally uniformly convergent series is given by :
	\begin{align}
	\prescript{\mathcal{A}}{a}I^{\alpha, \beta}_xu(x)=\sum_{n=0}^{\infty} a_n \Gamma(\beta n+\alpha) \prescript{RL}{a}I^{\alpha+\beta n}_xu(x)=\mathcal{A}_{\Gamma}\left(\prescript{RL}{a}I^{\beta}_x\right) \prescript{RL}{a}I^{\alpha}_x u(x).
	\end{align} 
\end{proposition}

\begin{proposition}[Semigroup property \cite{FOB}]\label{thm:spo}
We use all notation as in Definition~\ref{def:F}. For a semigroup property in one parameter, the equality
	\begin{align*}
\prescript{\mathcal{A}}{a}I_x^{\alpha, \beta}\;\circ\;\prescript{\mathcal{A}}{a}I_x^{\gamma, \beta}\; u(x)&=\;\prescript{\mathcal{A}}{a}I_x^{\alpha+\gamma, \beta} u(x)
	\end{align*}
is uniformly valid (i.e., for all $\alpha,\beta,\gamma,u$) if and only if the condition below is uniformly satisfied for all positive integers $k$:
	\begin{align*}
	a_k(\alpha+\gamma,\beta) \Gamma(\beta k+\alpha+\gamma)=\sum_{m+n=k}^{} a_n(\alpha,\beta) a_m(\gamma,\beta)\Gamma(\beta n+\alpha)\Gamma(\beta m+\gamma).
	\end{align*}
For a semigroup property in two parameters, the equality
	\begin{align*}
\prescript{\mathcal{A}}{a}I_x^{\alpha, \beta}	\;\circ\;\prescript{\mathcal{A}}{a}I_x^{\gamma, \delta}\; u(x)&=\prescript{\mathcal{A}}{a}I_x^{\alpha+\gamma, \beta+\delta} u(x),
	\end{align*}
cannot be uniformly valid (i.e., for all $\alpha,\beta,\gamma,\delta,u$).
\end{proposition}




\begin{definition}[Generalised fractional derivative: Riemann--Liouville sense \cite{FOB}] \label{Def:AKdR}
	Let $[a,b]$ be a real interval, $\alpha,\beta\in\mathbb{C}$ with $\Real(\alpha)\geq0$, $\Real(\beta)>0$, and $R\in\mathbb{R}^+$ satisfying $R > (b-a)^{\Real(\beta)}$. Let $\mathcal{A}$ be an analytic function as in Definition~\ref{def:F}. Then, the generalised fractional derivative in the Riemann--Liouville sense of a sufficiently differentiable function $u\in L^1[a,b]$ with kernel $\mathcal{A}$ is given by:
	\begin{align}
	\prescript{\mathcal{A}_R}{a}{D}_{x}^{\alpha,\beta}u(x)&=\frac{\mathrm{d}^n}{\mathrm{d}x^n}\left(\prescript{\mathcal{\overline{\mathcal{A}}}}{a}{I}^{n-\alpha, \beta}u(x)\right),\;\;\;n-1\leq\Real(\alpha) < n\in \mathbb{N}, \label{20}
	\end{align}  
	where the new analytic function	$\overline{\mathcal{A}}(x) = \sum_{n=0}^{\infty}\overline{a}_{n} x^{n}$ is defined so that $\mathcal{A}_{\Gamma}\cdot\mathcal{\overline{A}}_{\Gamma}=1$.
\end{definition}

\begin{definition}[Generalised fractional derivative: Caputo sense \cite{FOB}]
With all notation as in Definition \ref{Def:AKdR}, the generalised fractional derivative in the Caputo sense of a sufficiently differentiable function $u\in L^1[a,b]$ with kernel $\mathcal{A}$ is given by:
	\begin{align}
	\prescript{\mathcal{A}_C}{a}{D}_{x}^{\alpha,\beta}u(x)&=\prescript{\mathcal{\overline{A}}}{a}I^{n-\alpha,\beta}\left(\frac{\mathrm{d}^n}{\mathrm{d}x^n} u(x)\right),\;\;\;n-1\leq\Real(\alpha) < n\in \mathbb{N}. \label{21}
	\end{align}
\end{definition}

\section{Fractional calculus with analytic kernels with respect to functions}\label{chap:5}

Having established the fundamentals of two major general classes of fractional-calculus operators, we now turn our attention to the class resulting from the combination of these two ideas. Our starting point will be the definition, presented briefly in \cite{FOB}, of the general fractional integrals with analytic kernels with respect to functions. In \cite{FOB}, this definition was presented towards the end of the paper and was not elaborated on beyond some basic discussion. Here, therefore, we shall conduct a fuller analysis and create, for the first time, a proper theory of these generalised operators.

\subsection{Generalised fractional integrals with respect to functions}

\begin{definition}[Generalised fractional integral with respect to function \cite{FOB}]\label{Def:GF}
	Let $[a,b]$ be a real interval and $\alpha,\beta\in\mathbb{C}$ with $\Real(\alpha)>0$, $\Real(\beta)>0$. Let $\mathcal{A}$ be an analytic function as in Definition~\ref{def:F}, and let $\varphi\in C^1[a,b]$ be a monotonic function as in Definition \ref{Def:Iwrt}. Then, the generalised fractional integral of the function $u$ with respect to $\varphi$ is given by the following:
	\begin{align}
	\prescript{\mathcal{A}}{a}I^{\alpha, \beta}_{\varphi(x)}u(x)&:=\int_{a}^{x} \varphi'(t)[\varphi(x)-\varphi(t)]^{\alpha-1} \mathcal{A}[(\varphi(x)-\varphi(t))^{\beta}]u(t)\,\mathrm{d}t. \label{52}
	\end{align}
	(The appropriate function space for $u$ is in fact $L^1_{\varphi}[a,b]$, but this will be defined later, in Section \ref{chap:function} on function spaces and functional analysis of the considered operators.)
\end{definition}

The formula $\eqref{52}$ is the natural way of combining the formula \eqref{E} for fractional integrals with general analytic kernels with the formalism of Section \ref{Subsec:wrt} for taking fractional-calculus operators with respect to functions. As such, it generalises both classes of fractional calculus, and also other operators which do not fall into either class, as we show in the following examples.

\begin{example}
For $\varphi(x)=x$, we recover the fractional integral with analytic kernel $\mathcal{A}$ as seen in Section \ref{chap:4} above. Thus, the whole class of fractional calculus with general analytic kernels falls within our new general class.
\end{example}

\begin{example}
For $\beta=0$ and $\mathcal{A}(x)=\frac{1}{\Gamma(\alpha)}$, we have
\begin{align*}
\prescript{^\mathcal{A}}{a}I^{\alpha,\beta}_{\varphi(x)}u(x)&\coloneqq \int_{a}^{x} \varphi'(t)(\varphi(x)-\varphi(t))^{\alpha-1} \frac{1}{\Gamma(\alpha)}u(t) \,\mathrm{d}t=\prescript{RL}{a}I^{\alpha}_{\varphi(x)} u(x),
\end{align*}  
the Riemann--Liouville fractional integral of order $\alpha$ with respect to the function $\varphi$. Thus, the whole class of fractional calculus with respect to functions, as defined in Section \ref{Subsec:wrt}, falls within our new general class.
\end{example} 


\begin{example} \label{Ex:HT}
For $\beta=1$ and $\mathcal{A}(x)=
\frac{e^{-sx}}{\Gamma(\alpha)}$ and $\varphi(x)=\log x$, we have
\begin{align*}
\prescript{^\mathcal{A}}{a}I^{\alpha,1}_{\varphi(x)} u(x)&= \frac{1}{\Gamma(\alpha)} \int_{a}^{x} \frac{1}{t}\log x-\log t]^{\alpha-1} e^{-s(\log x-\log t)} u(t)\,\mathrm{d}t\\
&=\frac{1}{\Gamma(\alpha)} \int_{a}^{x} \left(\log\frac{x}{t} \right)^{\alpha-1}\left(\frac{t}{x}\right)^s\frac{u(t)}{t}\,\mathrm{d}t\\
&=\prescript{H}{a}{I}^{\alpha,1}_{\log x}u(x),
\end{align*}
which is the Hadamard-type fractional integral of order $\alpha$ and parameter $s$. Note that Hadamard-type fractional calculus does not fall within the class of operators with analytic kernels, nor within the class of operators with respect to functions. But since it is the same as tempered fractional calculus with respect to a logarithm function \cite{FFS}, it does fall within the generalised class which we are now studying.
\end{example}

\begin{theorem}[Series formula] 
\label{lem:series} 
We consider all notation as in Definition~\ref{Def:GF}. Then, the general fractional integral $\prescript{\mathcal{A}}{a}I^{\alpha,\beta}_{\varphi(x)} u(x)$, with analytic kernel $\mathcal{A}$ and with respect to a Lipschitz continuous function $\varphi$, is given by a locally uniformly convergent series as follows:
\begin{align}
\prescript{\mathcal{A}}{a}I^{\alpha,\beta}_{\varphi(x)} u(x)=\sum_{n=0}^{\infty} a_n \Gamma(\beta n+\alpha)\prescript{RL}{a}I^{\beta n+ \alpha}_{\varphi(x)} u(x)\label{53}
\end{align}
\end{theorem}

\begin{proof}
Starting from the definition \eqref{52}, we have
	\begin{align*}
		\prescript{\mathcal{A}}{a}I^{\gamma,\delta}_{\varphi(x)} u(x)&=\int_{a}^{x} \varphi'(t)(\varphi(x)-\varphi(t))^{\gamma-1}\mathcal{A}\left[(\varphi(x)-\varphi(t))^{\delta}\right]u(t)\,\mathrm{d}t\\
		&=\int_{a}^{x} \varphi'(t)(\varphi(x)-\varphi(t))^{\gamma-1}\sum_{n=0}^{\infty}\left[a_n(\varphi(x)-\varphi(t))^{\delta n}\right] u(t)\,\mathrm{d}t\\
		&=\int_{a}^{x} \varphi'(t)\sum_{n=0}^{\infty} a_n (\varphi(x)-\varphi(t))^{\delta n+\gamma-1} u(t)\,\mathrm{d}t.
	\end{align*}
	Now, since the function $\varphi$ is Lipschitz continuous on the analytic disc of radius $R$, then with the condition $a\leq t\leq x\leq b$, we have the above series which is locally uniformly convergent, since $0\leq |\varphi(x)-\varphi(t)|^{\delta}\leq |x-t|^{\delta}\leq (b-a)^{\Real(\delta)}<R$ and the series $\mathcal{A}$ is locally uniformly continuous on $\mathcal{D}(0,R)$. Therefore, we can interchange the integral and summation and obtain
	\begin{align*}
		\prescript{\mathcal{A}}{a}I^{\gamma,\delta}_{\varphi(x)} u(x) &=\sum_{n=0}^{\infty} \int_{a}^{x} a_n \varphi'(t)(\varphi(x)-\varphi(t))^{\delta n+\gamma-1} u(t)\,\mathrm{d}t \\
		&=\sum_{n=0}^{\infty} a_n \Gamma(\delta n+\gamma) \left[\frac{1}{\Gamma(\delta n+\gamma)}\int_{a}^{x}\varphi'(t)(\varphi(x)-\varphi(t))^{\delta n}u(t)\,\mathrm{d}t\right] \\
		&=\sum_{n=0}^{\infty} a_n \Gamma(\delta n+\gamma)\prescript{RL}{a}I^{\delta n+\gamma}_{\varphi(x)} u(x),	
	\end{align*}
	which ends the proof of the series formula.
\end{proof}

\begin{theorem}[Commutativity]
	\label{thm:com}
With all notation as in Definition~\ref{Def:GF}, the set given by 
	\begin{align*}
\left\{\prescript{\mathcal{A}}{a}I^{\alpha,\beta}_{\varphi(x)}\; :\; \alpha, \beta\in \mathbb{C}, \Real(\alpha)>0, Re(\beta)>0\right\}
	\end{align*}
	is a commutative family of operators.
\end{theorem}

\begin{proof}
	It suffices to show that for all $\alpha,\beta,\gamma,\delta\in\mathbb{C}$ with positive real part, we have
	\begin{align}
	\prescript{\mathcal{A}}{a}I^{\alpha,\beta}_{\varphi(x)}\circ \prescript{\mathcal{A}}{a}I^{\gamma,\delta}_{\varphi(x)} u(x)&=
	\prescript{\mathcal{A}}{a}I^{\gamma,\delta}_{\varphi(x)}\circ \prescript{\mathcal{A}}{a}I^{\alpha,\beta}_{\varphi(x)} u(x),
	\end{align}
regardless of whether or not there is a semigroup property to yield operators with $\alpha+\gamma$ and $\beta+\delta$. Using Theorem~\ref{lem:series}, we have
	\begin{align*}
	\prescript{\mathcal{A}}{a}I^{\alpha,\beta}_{\varphi(x)}\circ \prescript{\mathcal{A}}{a}I^{\gamma,\delta}_{\varphi(x)} u(x)&=\sum_{n=0}^{\infty} a_n \Gamma(\beta n +\alpha)
	\prescript{RL}{a}I^{\beta n+\alpha}_{\varphi(x)}\left[\sum_{m=0}^{\infty} a_m \Gamma(\delta m +\gamma)
	\prescript{RL}{a}I^{\delta m+\gamma}_{\varphi(x)}u(x)\right]\\
	&=\sum_{n,m}^{} a_n\;\Gamma(\beta n+\alpha)\; a_m\; \Gamma(\delta m+\gamma)
	\prescript{RL}{a}I^{\beta n+\alpha}_{\varphi(x)}\circ
	\prescript{RL}{a}I^{\delta m+\gamma}_{\varphi(x)} u(x).
	\end{align*}
	But we know from Proposition \ref{pro:a} that the fractional integral with respect to a function in the Riemann--Liouville sense is commutative, therefore
	\[
	\prescript{RL}{a}I^{\beta n+\alpha}_{\varphi(x)}\circ\prescript{RL}{a}I^{\delta m+\gamma}_{\varphi(x)} u(x)=\prescript{RL}{a}I^{\delta m+\gamma}_{\varphi(x)}\circ \prescript{RL}{a}I^{\beta n+\alpha}_{\varphi(x)} u(x).
	\]
	It follows that
	\begin{align*}
	\prescript{\mathcal{A}}{a}I^{\alpha,\beta}_{\varphi(x)}\circ
	\prescript{\mathcal{A}}{a}I^{\gamma,\delta}_{\varphi(x)} u(x)&=\sum_{m,n}^{} a_m\; \Gamma(\delta m+\gamma)\; a_n\;\Gamma(\beta n+\alpha)\prescript{RL}{a}I^{\delta m+\gamma}_{\varphi(x)}\circ\prescript{RL}{a}I^{\beta n+\alpha}_{\varphi(x)}u(x)\\
	&=\sum_{m=0}^{\infty} a_m \Gamma(\delta m +\gamma)\prescript{RL}{a}I^{\delta m+\gamma}_{\varphi(x)}\left[\sum_{n=0}^{\infty} a_n \Gamma(\beta n +\alpha)\prescript{RL}{a}I^{\beta n+\alpha}_{\varphi(x)} u(x)\right]\\
	&=\prescript{\mathcal{A}}{a}I^{\gamma, \delta}_{\varphi(x)}\circ\prescript{\mathcal{A}}{a}I^{\alpha,\beta}_{\varphi(x)}u(x).
	\end{align*}
	Thus the required commutativity result holds.	
\end{proof}

\begin{theorem}[Semigroup property in one parameter]
	\label{thm:spop}
	We consider all notation as in Definition~\ref{Def:GF}. Then, the following equality:
	\begin{align}
	\prescript{\mathcal{A}}{a}I^{\alpha,\beta}_{\varphi(x)}\circ\prescript{\mathcal{A}}{a}I^{\gamma,\beta}_{\varphi(x)}u(x)&=\prescript{\mathcal{A}}{a}I^{\alpha+\gamma,\beta}_{\varphi(x)}u(x)
	\end{align}
	is uniformly valid (i.e., for all $\alpha,\beta,\gamma,u$) if and only if the condition below is satisfied for any non-negative integer $k$:
	\begin{align*}
	a_k(\alpha+\gamma,\beta) \Gamma(\beta k+\alpha+\gamma)=\sum_{m+n=k}^{} a_n(\alpha,\beta) a_m(\gamma,\beta)\Gamma(\beta n+\alpha)\Gamma(\beta m+\gamma).
	\end{align*}
\end{theorem}

\begin{proof}
Using the series formula from Theorem~\ref{lem:series}, and the semigroup property for Riemann--Liouville fractional integrals with respect to a function from Proposition~\ref{pro:a}, we have:
	\begin{align}
	\prescript{\mathcal{A}}{a}I^{\alpha,\beta}_{\varphi(x)}\circ
	\prescript{\mathcal{A}}{a}I^{\gamma,\beta}_{\varphi(x)}u(x)&=\sum_{n=0}^{\infty} a_n \Gamma(\beta n +\alpha)\prescript{RL}{a}I^{\beta n+\alpha}_{\varphi(x)}\left[\sum_{m=0}^{\infty} a_m \Gamma(\beta m +\gamma)\prescript{RL}{a}I^{\beta m+\gamma}_{\varphi(x)} u(x)\right] \nonumber \\
	&=\sum_{n,m}^{} a_n\;\Gamma(\beta n+\alpha)\; a_m\; \Gamma(\beta m+\gamma)\prescript{RL}{a}I^{\beta n+\alpha}_{\varphi(x)}\circ\prescript{RL}{a}I^{\beta m+\gamma}_{\varphi(x)} u(x)\nonumber \\
	&=\sum_{n,m}^{} a_n\;\Gamma(\beta n+\alpha)\; a_m\; \Gamma(\beta m+\gamma)\prescript{RL}{a}I^{\alpha+\gamma+\beta n+\beta m}_{\varphi(x)} u(x)\nonumber\\
	&=\sum_{k=0}^{\infty} \left[\sum_{m+n=k}^{} a_n a_m \Gamma(\beta n+\alpha)\Gamma(\beta m+\gamma)\right]\prescript{RL}{a}I^{\alpha+\gamma k+\beta}_{\varphi(x)} u(x),\label{54}
	\end{align}
	where in the last line we set $k=m+n$.
	
	On the other hand, directly from Theorem~\ref{lem:series} and replacing $\alpha$ by $\alpha+\gamma$, we have 
	\begin{align}
	\prescript{\mathcal{A}}{a}I^{\alpha+\gamma,\beta}_{\varphi(x)} u(x)&=\sum_{k=0}^{\infty} a_k \Gamma(\beta k+\alpha+\gamma)
	\prescript{RL}{a}I^{\alpha+\gamma+\beta k}_{\varphi(x)}u(x).\label{55}	
	\end{align}
	Therefore, it can be seen that the right hand sides of \eqref{54} and \eqref{55} are uniformly equal if and only if the following holds:
	\begin{align*}
	a_k(\alpha+\gamma,\beta) \Gamma(\beta k+\alpha+\gamma)=\sum_{m+n=k}^{} a_n(\alpha,\beta) a_m(\gamma,\beta)\Gamma(\beta n+\alpha)\Gamma(\beta m+\gamma).
	\end{align*}
	for all $k=0,1,2,...$. This completes the proof.
\end{proof}

\begin{theorem}[Semigroup property in two parameters]
	\label{thm:sp2p}
	We consider all notation as in Definition~\ref{Def:GF}. Then, the following equality:
	\begin{align}
	\prescript{\mathcal{A}}{a}I^{\alpha,\beta}_{\varphi(x)}\circ\prescript{\mathcal{A}}{a}I^{\gamma, \delta}_{\varphi(x)} u(x)&=\prescript{\mathcal{A}}{a}I^{\alpha+\gamma,\beta+\delta}_{\varphi(x)}u(x)
	\end{align}
cannot be uniformly valid (i.e., for all $\alpha,\beta,\gamma,\delta,u$).
\end{theorem}

\begin{proof}
Using the series formula from Theorem~\ref{lem:series}, and the semigroup property for Riemann--Liouville fractional integrals with respect to a function from Proposition~\ref{pro:a}, we have:
	\begin{align}
	\prescript{\mathcal{A}}{a}I^{\alpha,\beta}_{\varphi(x)}\circ\prescript{\mathcal{A}}{a}I^{\gamma,\delta}_{\varphi(x)} u(x)
	&=\sum_{n=0}^{\infty} a_n \Gamma(\beta n +\alpha)\prescript{RL}{a}I^{\beta n+\alpha}_{\varphi(x)} \left[\sum_{m=0}^{\infty} a_m \Gamma(\delta m +\gamma)\prescript{RL}{a}I^{\delta m+\gamma}_{\varphi(x)}u(x)\right] \nonumber \\
	&=\sum_{n,m}^{} a_n\;\Gamma(\beta n+\alpha)\; a_m\; \Gamma(\delta m+\gamma)\prescript{RL}{a}I^{\beta n+\alpha}_{\varphi(x)}\circ\prescript{RL}{a}I^{\delta m+\gamma}_{\varphi(x)}u(x)\nonumber \\
	&=\sum_{n,m}^{} a_n\;\Gamma(\beta n+\alpha)\; a_m\; \Gamma(\delta m+\gamma)\prescript{RL}{a}I^{\alpha+\gamma+\beta n+\delta m}_{\varphi(x)}u(x)\nonumber\\
	&=\sum_{k=0}^{\infty} a_k a_k\Gamma(\beta k+\gamma)\Gamma(\delta k+\gamma)\prescript{RL}{a}I^{\alpha+\gamma+k(\beta+\delta)}_{\varphi(x)}u(x) \nonumber\\
	&\hspace{2cm}+\sum_{m\neq n}^{} a_n \Gamma(\beta n+\alpha) a_m \Gamma(\delta m+\gamma)\prescript{RL}{a}I^{\alpha+\gamma+\beta n+\delta m}_{\varphi(x)}u(x),\label{56}
	\end{align}
	where we have split the sum into one part with $(\beta+\delta)$ times an integer in the fractional order and another part with $\beta$ and $\delta$ times different integers.
	
	On the other hand, directly from Theorem~\ref{lem:series} and replacing $\alpha$ by $\alpha+\gamma$ and $\beta$ by $\beta+\delta$, we have
	\begin{align}
	\prescript{\mathcal{A}}{a}I^{\alpha+\gamma,\beta+\delta}_{\varphi(x)} u(x)&=\sum_{k=0}^{\infty} a_k \Gamma((\beta+\delta) k+\alpha+\gamma)\prescript{RL}{a}I^{\alpha+\gamma+k(\beta+\delta)}_{\varphi(x)}.\label{57}	
	\end{align}
	It follows that we have uniform equality between the right hand sides of \eqref{56} and \eqref{57} if and only if the following two identities hold:
	\begin{align*}
	a_k(\alpha,\beta) a_k(\gamma, \delta) \Gamma(\beta k+\alpha)\Gamma(\delta k+\gamma)=a_k(\alpha+\beta,\gamma+\delta)\Gamma((\beta+\delta)k+\alpha+\gamma),\;\;\; k\in \mathbb{Z}_0^+.
	\end{align*}
	and
	\begin{align*}
	a_n(\alpha,\beta)\Gamma(\beta n+\alpha)a_m(\gamma,\delta)\Gamma(\delta m+\gamma)=0,\;\;\;\;\;\; \text{for\;all}\;\;m,n\in \;\mathbb{Z}_0^+, \text{with}\;\; m\neq n.
	\end{align*}
	But, just like in \cite[Theorem 2.12]{FOB}, this latter condition cannot be uniformly valid since it implies $\mathcal{A}$ would be identically zero.
\end{proof}

\begin{remark}
Note that the results and derivations of Theorems \ref{thm:spop} and \ref{thm:sp2p} do not rely at all on the function $\varphi$: the same semigroup properties hold for any choice of $\varphi$. This is because, as we saw in Section \ref{Subsec:wrt} above, all composition properties for fractional calculus with respect to functions are the same as the corresponding properties for the original fractional calculus. Thanks to the conjugation relations of Proposition \ref{pro:x1}, we know that composition and semigroup properties do not depend on the choice of $\varphi$, only on the choice of $\mathcal{A}$.
\end{remark}

\subsection{Generalised fractional derivatives with respect to functions}

In the paper \cite{FOB} where fractional calculus with general analytic kernels was first theorised, the integral operator was defined first, and the derivative operators came only after lengthy discussion about how exactly they should be defined to make the model most natural and sensible. In the final section of that paper which introduced fractional calculus with analytic kernels with respect to functions, only the fractional integral operator was mentioned in this model. The corresponding fractional derivative operators, therefore, are a new contribution of our current work, albeit a natural one given the fractional derivatives with analytic kernels as defined in \cite{FOB} and the usual way of extending fractional operators to be taken with respect to functions.

\begin{definition}[Generalised fractional derivative with respect to function: Riemann--Liouville sense]\label{Def:GD}
	Let $[a,b]$ be a real interval and $\alpha,\beta\in\mathbb{C}$ with $\Real(\alpha)\geq0$, $\Real(\beta)>0$. Let $\mathcal{A}$ be an analytic function as in Definition~\ref{def:F}, and let $\varphi\in C^n[a,b]$ be a monotonic function as in Definition \ref{Def:Dwrt}. Then, the generalised fractional derivative (in the Riemann--Liouville sense) of a function $u\in L^1[a,b]$ with sufficient differentiability properties with respect to $\varphi$ is given by the following:
	\begin{align}
	\prescript{\mathcal{A}_R}{a}D^{\alpha,\beta}_{\varphi(x)} u(x)&=\left(\frac{1}{\varphi'(x)}\cdot\frac{\mathrm{d}}{\mathrm{d}x}\right)^m\prescript{\mathcal{\overline{A}}}{a}I^{m-\alpha,\beta}_{\varphi(x)} u(x),\;\;\;\text{with}\;\;m-1< \Real(\alpha) < m\in \mathbb{N}, \label{58}
	\end{align}
	where the new analytic function $\overline{\mathcal{A}}$ is defined so that $\mathcal{A}_{\Gamma}\cdot\mathcal{\overline{A}}_{\Gamma}=1$.
\end{definition}

\begin{definition}[Generalised fractional derivative with respect to function: Caputo sense]\label{Def:GDC}
	Let $[a,b]$ be a real interval and $\alpha,\beta\in\mathbb{C}$ with $\Real(\alpha)\geq0$, $\Real(\beta)>0$. Let $\mathcal{A}$ be an analytic function as in Definition~\ref{def:F}, and let $\varphi\in C^n[a,b]$ be a monotonic function as in Definition \ref{Def:Dwrt}. Then, the generalised fractional derivative in the Caputo sense of a function $u\in L^1[a,b]$ with sufficient differentiability properties with respect to $\varphi$ is given by the following:
	\begin{align}
	\prescript{\mathcal{A}_C}{a}D^{\alpha,\beta}_{\varphi(x)} u(x)&=\prescript{\mathcal{\overline{A}}}{a}I^{m-\alpha,\beta}_{\varphi(x)}\left(\frac{1}{\varphi'(x)}\cdot\frac{\mathrm{d}}{\mathrm{d}x}\right)^m u(x),\;\;\;\text{with}\;\;m-1< \Real(\alpha) < m\in \mathbb{N},
	\end{align}
	where the new analytic function $\overline{\mathcal{A}}$ is defined so that $\mathcal{A}_{\Gamma}\cdot\mathcal{\overline{A}}_{\Gamma}=1$.
\end{definition}

\begin{theorem}[Series formula]
	\label{lem:Sf}
With all notation as in Definition~\ref{Def:GD}, the general fractional derivative $\prescript{\mathcal{A}_R}{a}D^{\alpha,\beta}_{\varphi(x)} u(x)$, with analytic kernel $\mathcal{A}$ and with respect to a Lipschitz continuous function $\varphi$, is given by a locally uniformly convergent series as follows:
	\begin{align}
	\prescript{\mathcal{A}}{a}D^{\alpha,\beta}_{\varphi(x)}u(x)&=\sum_{n=0}^{\infty} \overline{a}_n \Gamma(\beta n-\alpha+m)\prescript{RL}{a}D^{\alpha-\beta n}_{\varphi(x)}u(x), \;\; m,n\in \mathbb{N},\label{59}
	\end{align}
	where as usual $\overline{\mathcal{A}}(x)=\sum_{n=0}^{\infty}\overline{a}_nx^n$ is defined so that $\mathcal{A}_{\Gamma}\cdot\mathcal{\overline{A}}_{\Gamma}=1$, and $\overline{\mathcal{A}}$ is assumed to be analytic on the disc $\mathcal{D}(0,R)$ mentioned in Definition \ref{def:F}, and where $m=\lfloor\Real(\alpha)\rfloor+1$ as in Definition \ref{Def:GD}.
\end{theorem}

\begin{proof}
Using the result of Theorem \ref{lem:series}, we have
	\begin{align*}
	\prescript{\mathcal{A}}{a}D^{\alpha,\beta}_{\varphi(x)}	u(x)&=\left(\frac{1}{\varphi'(x)}\cdot\frac{\mathrm{d}}{\mathrm{d}x}\right)^m\prescript{\mathcal{\overline{A}}}{a}I^{m-\alpha, \beta}_{\varphi(x)}u(x) \\
	&=\left(\frac{1}{\varphi'(x)}\cdot\frac{\mathrm{d}}{\mathrm{d}x}\right)^m\sum_{n=0}^{\infty} \overline{a}_n \Gamma(\beta n-\alpha+m)\prescript{RL}{a}I^{\beta n-\alpha+m}_{\varphi(x)}u(x) \\
	&=\sum_{n=0}^{\infty} \overline{a}_n \Gamma(\beta n-\alpha+m) \left(\frac{1}{\varphi'(x)}\cdot\frac{\mathrm{d}}{\mathrm{d}x}\right)^m\prescript{RL}{a}I^{\beta n-\alpha+m}_{\varphi(x)}u(x) \\
	&=\sum_{n=0}^{\infty} \overline{a}_n \Gamma(\beta n-\alpha+m)\prescript{RL}{a}D^{\alpha-\beta n}_{\varphi(x)}u(x),
	\end{align*}
	where in the last step we used the composition property of Riemann--Liouville differintegrals with respect to functions given by Proposition \ref{pro:b}. In order to justify exchanging the order of differintegration and summation in the manipulation above, according to standard theorems on termwise differintegration of series, we needed local uniform convergence of both the pre-differentiation series and the post-differentiation series. The pre-differentiation series of Riemann--Liouville integrals is locally uniformly convergent according to Theorem \ref{lem:series}, and the post-differentiation series also consists only of Riemann--Liouville integrals for sufficiently large $n$ in the sum:
	\begin{multline*}
	\sum_{n=N}^{\infty} \overline{a}_n \Gamma(\beta n-\alpha+m)\prescript{RL}{a}D^{\alpha-\beta n}_{\varphi(x)}u(x) \\ =\sum_{n=0}^{\infty} \overline{a}_{n+N} \frac{\Gamma(\beta n-\alpha+m)}{\Gamma(\beta n+(\beta N-\alpha))}\Gamma(\beta n+(\beta N-\alpha))\prescript{RL}{a}I^{\beta n+(\beta N-\alpha)}_{\varphi(x)}u(x),
	\end{multline*}
where $N$ is large enough that $\Real(\beta N-\alpha)>0$ so that this series is a case of the one in Theorem \ref{lem:series}, a locally uniformly convergent series of Riemann--Liouville integrals, with an extra multiplier consisting of a quotient of two gamma functions, which does not affect convergence by \cite[Theorem IV.1]{MillerRoss}. Therefore, this series is also locally uniformly convergent, and the proof is complete.
\end{proof}

\begin{theorem} \label{pro:1}
We consider all notation as in Definition~\ref{Def:GD}. Then, we have the following identities valid for all $m\in\mathbb{N}$ and $\alpha\in\mathbb{C}$ with $\Real(\alpha)>0$:
	\begin{align*}
	\prescript{RL}{a}I_{\varphi(x)}^m\left(\prescript{\mathcal{A}}{a}I^{\alpha,\beta}_{\varphi(x)}u(x)\right)&=\prescript{\mathcal{B}}{a}I^{m+\alpha,\beta}_{\varphi(x)}u(x), \\
	\left(\frac{1}{\varphi'(x)}\cdot\frac{\mathrm{d}}{\mathrm{d}x}\right)^m\left(\prescript{\mathcal{A}}{a}I^{m+\alpha,\beta}_{\varphi(x)}u(x)\right)&=\prescript{\mathcal{C}}{a}I^{\alpha,\beta}_{\varphi(x)}u(x),
	\end{align*}
	where the modified analytic functions $\mathcal{B}$ and $\mathcal{C}$ are given by
	\begin{align*}
	\mathcal{B}(x)&=\sum_{n=0}^{\infty}\frac{a_n}{(\beta n+\alpha)_m} x^n=\sum_{n=0}^{\infty}\frac{\Gamma(\beta n+\alpha)}{\Gamma(\beta n+\alpha+m)} a_n x^n, \\
	\mathcal{C}(x)&=\sum_{n=0}^{\infty}(\beta n+\alpha)_m a_n x^n=\sum_{n=0}^{\infty}\frac{\Gamma(\beta n+\alpha+m)}{\Gamma(\beta n+\alpha)} a_n x^n,
	\end{align*}
both of these series being locally uniformly convergent, just like the one for $\mathcal{A}$, according to \cite[Theorem IV.1]{MillerRoss}.
\end{theorem}

\begin{proof}
Using the series formula of Theorem \ref{lem:series}, and composition properties of fractional operators with respect to functions (Propositions \ref{pro:a} and \ref{pro:b}), we have:
	\begin{align*}
	\prescript{RL}{a}I_{\varphi(x)}^m\left(\prescript{\mathcal{A}}{a}I^{\alpha,\beta}_{\varphi(x)}u(x)\right)&=\prescript{RL}{a}I_{\varphi(x)}^m\left(\sum_{n=0}^{\infty} a_n \Gamma(\beta n+\alpha)\prescript{RL}{a}I^{\beta n+\alpha}_{\varphi(x)}u(x)\right) \\
	&=\sum_{n=0}^{\infty} a_n \Gamma(\beta n+\alpha) \prescript{RL}{a}I_{\varphi(x)}^m \prescript{RL}{a}I^{\beta n+\alpha}_{\varphi(x)}u(x) \\
	&=\sum_{n=0}^{\infty} a_n \Gamma(\beta n+\alpha)\prescript{RL}{a}I^{\beta n+\alpha+m}_{\varphi(x)}u(x) \\
	&=\sum_{n=0}^{\infty} \left[a_n\frac{\Gamma(\beta n+\alpha)}{\Gamma(\beta n+\alpha+m)}\right] \Gamma(\beta n+\alpha+m)\prescript{RL}{a}I^{\beta n+\alpha+m}_{\varphi(x)}u(x) \\
	&=\prescript{\mathcal{B}}{a}I^{m+\alpha,\beta}_{\varphi(x)}u(x),
	\end{align*}
and
	\begin{align*}
	\left(\frac{1}{\varphi'(x)}\cdot\frac{\mathrm{d}}{\mathrm{d}x}\right)^m\left(\prescript{\mathcal{A}}{a}I^{m+\alpha,\beta}_{\varphi(x)}u(x)\right)&=\left(\frac{1}{\varphi'(x)}\cdot\frac{\mathrm{d}}{\mathrm{d}x}\right)^m \left(\sum_{n=0}^{\infty} a_n \Gamma(\beta n+\alpha+m)\prescript{RL}{a}I^{\beta n+\alpha+m}_{\varphi(x)}u(x)\right) \\
	&=\sum_{n=0}^{\infty} a_n \Gamma(\beta n+\alpha) \left(\frac{1}{\varphi'(x)}\cdot\frac{\mathrm{d}}{\mathrm{d}x}\right)^m	\prescript{RL}{a}I^{\beta n+\alpha}_{\varphi(x)}u(x) \\
	&=\sum_{n=0}^{\infty} a_n \Gamma(\beta n+\alpha+m)\prescript{RL}{a}I^{\beta n+\alpha}_{\varphi(x)}u(x) \\
	&=\sum_{n=0}^{\infty} \left[a_n\frac{\Gamma(\beta n+\alpha+m)}{\Gamma(\beta n+\alpha)}\right] \Gamma(\beta n+\alpha)\prescript{RL}{a}I^{\beta n+\alpha}_{\varphi(x)}u(x) \\
	&=\prescript{\mathcal{C}}{a}I^{\alpha,\beta}_{\varphi(x)}u(x),
	\end{align*}
	where in each case we exchange the order of differintegral and sum since the relevant series are convergent.
\end{proof}

\begin{theorem} \label{pro:com}
We consider all notation as in Definition~\ref{Def:GD} and $\alpha,\beta,\gamma\in \mathbb{C}$ with positive real part and $\Real(\gamma)<\Real(\alpha)$. Then, the following composition relation:
	\begin{align*}
	\prescript{\mathcal{A}}{a}D^{\alpha,\beta}_{\varphi(x)}\circ\prescript{\mathcal{A}}{a}I^{\gamma,\beta}_{\varphi(x)}u(x)&=
	\prescript{\mathcal{A}}{a}D^{\alpha-\gamma,\beta}_{\varphi(x)}u(x)
	\end{align*}
is valid if and only if we have the following equality for all $k\in\mathbb{Z}^+_0$:
	\begin{align*}
	\overline{a}_k(\alpha-\gamma,\beta)\Gamma(\beta k-\alpha+\gamma+m)=\sum_{n+p=k}\overline{a}_n(\alpha, \beta) a_p(\gamma, \beta) \Gamma(\beta n-\alpha+m)\Gamma(\beta p+\gamma),
	\end{align*}
where $m=\lfloor\Real(\alpha)\rfloor+1$ as in Definition \ref{def:F}.
\end{theorem}

\begin{proof}
Using the series formulae for $\prescript{\mathcal{A}}{a}D^{\alpha,\beta}_{\varphi(x)}$ and  $\prescript{\mathcal{A}}{a}I^{\gamma,\beta}_{\varphi(x)}$ (Theorem \ref{lem:series} and Theorem \ref{lem:Sf}), as well as the semigroup property for fractional calculus with respect to functions (specifically Proposition \ref{pro:c}), we have:
	\begin{align}	
	\prescript{\mathcal{A}}{a}D^{\alpha,\beta}_{\varphi(x)}\circ \prescript{\mathcal{A}}{a}I^{\gamma,\beta}_{\varphi(x)}u(x)&=\left[\sum_{n=0}^{\infty} \overline{a}_n \Gamma(\beta n-\alpha+m) \prescript{RL}{a}D^{\alpha-\beta n}_{\varphi(x)}\right]\circ \prescript{\mathcal{A}}{a}I^{\gamma,\beta}_{\varphi(x)}u(x),\nonumber\\
	&=\sum_{n=0}^{\infty} \overline{a}_n \Gamma(\beta n-\alpha+m) \prescript{RL}{a}D^{\alpha-\beta n}_{\varphi(x)}\circ \left[\sum_{p=0}^{\infty} a_p \Gamma(\beta p+\gamma) \prescript{RL}{a}I^{\beta p+\gamma}_{\varphi(x)}u(x)\right],\nonumber\\
	&=\sum_{n, p}^{} \overline{a}_n a_p \Gamma(\beta n-\alpha+m)	\Gamma(\beta p+\gamma) \prescript{RL}{a}D^{\alpha-\beta n}_{\varphi(x)} \circ \prescript{RL}{a}I^{\beta p+\gamma}_{\varphi(x)} u(x),\nonumber\\
	&=\sum_{n, p}^{} \overline{a}_n a_p \Gamma(\beta n-\alpha+m)	\Gamma(\beta p+\gamma) \prescript{RL}{a}D^{\alpha-\beta n-\beta p-\gamma}_{\varphi(x)} u(x),\nonumber\\
	&=\sum_{n, p}^{} \overline{a}_n a_p \Gamma(\beta n-\alpha+m)	\Gamma(\beta p+\gamma) \prescript{RL}{a}D^{\alpha-\gamma-\beta(n+p)}_{\varphi(x)} u(x),\nonumber\\
	&=\sum_{k=0}^{\infty} \left[\sum_{n+p=k}\overline{a}_n a_p \Gamma(\beta n-\alpha+m) \Gamma(\beta p+\gamma) \prescript{RL}{a}D^{\alpha-\gamma-\beta k}_{\varphi(x)}u(x)\right],\label{o}
	\end{align}	
where the coefficients $\overline{a}_n$ and $a_p$ are dependent on parameters as $\overline{a}_n(\alpha,\beta)$ and $a_p(\gamma,\beta)$.
	
	Meanwhile the series formula in Theorem~\ref{lem:Sf} yields
\begin{align}
\prescript{\mathcal{A}}{a}D^{\alpha-\gamma,\beta}_{\varphi(x)}u(x)=\sum_{k=0}^{\infty} \overline{a}_k \Gamma(\beta k-\alpha+\gamma+m)
\prescript{RL}{a}D^{\alpha-\gamma-\beta k}_{\varphi(x)}, \label{y}
\end{align}
where the coefficients $\overline{a}_k$ are dependent on parameters as $\overline{a}_k(\alpha-\gamma,\beta)$.

	Therefore the right hand sides of~\eqref{o} and~\eqref{y} are equal if and only if
	\begin{align*}
	\overline{a}_k(\alpha-\gamma,\beta)\Gamma(\beta k-\alpha+\gamma+m)=\sum_{n+p=k}^{} \overline{a}_n(\alpha, \beta) a_{p}(\gamma, \beta) \Gamma(\beta n-\alpha+m)\Gamma(\beta p+\gamma)
	\end{align*}
	holds for all values of $k$.
\end{proof}

\section{Functional analysis of the new generalised fractional calculus} \label{chap:function}

In the above discussion, several new operators have been defined and their properties have been analysed, however no mention has been made of the function spaces on which these operators act. It is important to know which types of functions can have such operators applied to them, in order that we can know how wide-ranging the new theory actually is. It turns out that the appropriate spaces to consider are weighted $L^p$ spaces, since $\varphi'(x)\,\mathrm{d}x$ is equivalent to a new measure $\mathrm{d}\varphi(x)$.

\subsection{Function spaces for the generalised fractional integral operator}

In this subsection, we establish an appropriate function space for the generalised fractional integral operator $\prescript{\mathcal{A}}{a}I^{\alpha, \beta}_{\varphi(x)}$ with respect to a function. Specifically, we prove that for any $p\in[1,\infty)$ this operator is bounded in the weighted $L^p$ space $L_{\varphi}^{p} [a, b]$, which we define as follows. This will justify the remark made during Definition \ref{Def:GF} about the weighted $L^1$ function space for this operator.

\begin{definition} \label{Genspace}
	Let $1 \leq  p <  \infty$. The space $L_{\varphi}^{p} [a, b]$ is defined to consist of all Lebesgue measurable functions $u:[a,b]\rightarrow\mathbb{C}$ for which $ \|u\|_{L_{\varphi}^{p}} < \infty$, where the norm is defined by
	\begin{equation} \label{generalizednorm}
		\|u\|_{L_{\varphi}^{p}} = \left( \int_{a}^{b} \left| u(x)\right|^{p} \varphi'(x) \,\mathrm{d}x \right)^{\frac{1}{p}}.
	\end{equation}
\end{definition}

If we consider $ \varphi(x)=x $, then the space $L_{\varphi}^{p} [a, b]$ coincides with the space $ L^{p}[a,b] $ with the usual norm
\[
\|u\|_{p} = \left( \int_{a}^{b} \left|u(x)\right|^{p}  dx \right)^{\frac{1}{p}} \ \ \ \text{for }1 \leq  p <  \infty.
\]

In the theorem below, we prove that, for a non-negative increasing function $ \varphi $, the generalised fractional integral operator of any real-part-positive order $\alpha$ with respect to $\varphi$ is well-defined on the weighted space $L_{\varphi}^{p} [a, b]$.

\begin{theorem}\label{thm4.1}
	With all notation as in Definition \ref{Def:GF} and \ref{Genspace}, we have a well-defined bounded operator $\prescript{\mathcal{A}}{a}I^{\alpha, \beta}_{\varphi(x)}$ from the space $L_{\varphi}^{p}[a, b]$ to itself, with
	\begin{equation}
		\left\|  	\prescript{\mathcal{A}}{a}I^{\alpha, \beta}_{\varphi(x)} u \right\|_{L_{\varphi}^{p}} \leq K \|u\|_{L_{\varphi}^{p}},
	\end{equation}
	where the constant $K$ is defined by
	\begin{equation} \label{Knorm}
		K = \frac{\left(\varphi(b)-\varphi(a)\right)^{\alpha}}{\alpha} \sup_{\left| \tau \right|< \left(  \varphi(b)-\varphi(a) \right)^{\beta} } \left| \mathcal{A}(\tau) \right|.
	\end{equation}
\end{theorem}

\begin{proof}
	Using the definition of $	\prescript{\mathcal{A}}{a}I^{\alpha, \beta}_{\varphi(x)}$ and Eq. \eqref{generalizednorm}, we find
	\[
	\left\| 	\prescript{\mathcal{A}}{a}I^{\alpha, \beta}_{\varphi(x)} u\right\|_{L_{\varphi}^{p}} =  \left( \int_{a}^{b} \left| \int_{a}^{x} \varphi'(t)[\varphi(x)-\varphi(t)]^{\alpha-1} \mathcal{A}[(\varphi(x)-\varphi(t))^{\beta}]u(t)\,\mathrm{d}t\right|^{p} \varphi'(x) \,\mathrm{d}x \right)^{\frac{1}{p}}.
	\]
	Making the substitution $\varphi(\tau) =\varphi(x)-\varphi(t) $ in the inner integral, which gives $\varphi'(\tau)\,\mathrm{d}\tau =-\varphi'(t)\,\mathrm{d}t$, we get	
	\begin{multline*}
		\left\| \prescript{\mathcal{A}}{a}I^{\alpha, \beta}_{\varphi(x)} u\right\|_{L_{\varphi}^{p}} \\
		= \left( \int_{a}^{b} \left| \int_{0}^{\varphi^{-1} \left( \varphi(x)-\varphi(a) \right)} \varphi'(\tau)\varphi(\tau)^{\alpha-1} \mathcal{A}[\varphi(\tau)^{\beta}]u\left( \varphi^{-1} \left( \varphi(x)-\varphi(\tau) \right) \right)\,\mathrm{d}\tau\right|^{p} \varphi'(x) \,\mathrm{d}x \right)^{\frac{1}{p}} \\
		=  \left( \int_{a}^{b} \left| \int_{0}^{\varphi^{-1} \left( \varphi(x)-\varphi(a) \right)} \varphi'(\tau)\varphi(\tau)^{\alpha-1} \mathcal{A}[\varphi(\tau)^{\beta}]{\varphi'(x)}^{\frac{1}{p}}u\left( \varphi^{-1} \left( \varphi(x)-\varphi(\tau) \right) \right)\,\mathrm{d}\tau\right|^{p}  \,\mathrm{d}x \right)^{\frac{1}{p}}.
	\end{multline*}
	Since $ u(x) \in L_{\varphi}^{p}[a,b]$, it follows that $ {\varphi'(x)}^{\frac{1}{p}}u(x) \in  L^{p}(a,b)$, and hence by application of the generalised Minkowski inequality, we have
	\begin{align*}
		\left\| \prescript{\mathcal{A}}{a}I^{\alpha, \beta}_{\varphi(x)} u\right\|_{L_{\varphi}^{p}} &\leq  \int_{0}^{\varphi^{-1} \left( \varphi(b)-\varphi(a) \right)}  \left( \int_{\varphi^{-1} \left( \varphi(\tau)+\varphi(a) \right)}^{b} \left| u\left( \varphi^{-1} \left( \varphi(x)-\varphi(\tau) \right) \right) \right|^{p} \varphi'(x)\,\mathrm{d}x \right)^{\frac{1}{p}} \\ &\hspace{7cm} \times\varphi'(\tau)\varphi(\tau)^{\alpha-1} \mathcal{A}[\varphi(\tau)^{\beta}] \,\mathrm{d}\tau \\
		&\hspace{-2cm}= \int_{0}^{\varphi^{-1} \left( \varphi(b)-\varphi(a) \right)}  \left( \int_{a}^{\varphi^{-1} \left( \varphi(b)-\varphi(\tau) \right)} \left| u\left( t \right) \right|^{p} \varphi'(t)\,\mathrm{d}t \right)^{\frac{1}{p}}   \varphi'(\tau)\varphi(\tau)^{\alpha-1} \mathcal{A}[\varphi(\tau)^{\beta}] \,\mathrm{d}\tau.
	\end{align*}
	Because $ \mathcal{A} $ is analytic on $\mathcal{D}(0,R)$ by definition, the function $  \mathcal{A}[\varphi(\tau)^{\beta}] $ is bounded on the finite interval $ [0,\varphi^{-1} \left( \varphi(b)-\varphi(a) \right)] $. Moreover, since $ u $ is in the space $L_{\varphi}^{p}[a, b]$, we get
	\begin{equation*}
		\left\| \prescript{\mathcal{A}}{a}I^{\alpha, \beta}_{\varphi(x)} u \right\|_{L_{\varphi}^{p}} \leq K \|u\|_{L_{\varphi}^{p}},
	\end{equation*}
	where
	\begin{align*} 
		K &\coloneqq \sup_{\left| \tau \right|< \left(  \varphi(b)-\varphi(a) \right)^{\beta} } \left| \mathcal{A}(\tau) \right|\int_{0}^{\varphi^{-1} \left( \varphi(b)-\varphi(a) \right)} \varphi'(\tau)\varphi(\tau)^{\alpha-1} \,\mathrm{d}\tau \\
		&=\frac{\left(\varphi(b)-\varphi(a)\right)^{\alpha}}{\alpha} \sup_{\left| \tau \right|< \left(  \varphi(b)-\varphi(a) \right)^{\beta} } \left| \mathcal{A}(\tau) \right|,
	\end{align*}
	and this completes the proof.
\end{proof}

\subsection{Function spaces for the generalised fractional derivative operator}

In this subsection, we establish an appropriate function space for the generalised fractional derivative operator $\prescript{\mathcal{A}_R}{a}D^{\alpha,\beta}_{\varphi(x)}$ with respect to a function. In order to do this, we must define a function space $C_{\alpha,\varphi}^{} [a, b]$ as follows, a weighted version of the function space $C_{\alpha}$ which is often used in operational calculus \cite{luchko, Luchko2,Srivastava}.

\begin{definition}\label{weightedspaces}
	Let $ [a,b] $ $ (0<a<b<\infty) $ be a finite interval on the half-axis $ \R^{+} $. Then, the weighted space $C_{\alpha,\varphi}^{} [a, b]$ of functions $ u $ on $ (a,b] $ is defined by
	\begin{equation*} \label{generalizedweightedspace1}
		C_{\alpha,\varphi}^{} [a, b] = \{ u : (a,b] \to \R\;\; :\;\; [\varphi(x)-\varphi(a)]^{\alpha} u(x) \in C[a,b] \},
	\end{equation*}
	with the norm 
	\begin{equation}\label{Calphaphinorm}
		\|u\|_{C_{\alpha,\varphi}} = \max_{x \in [a,b]} \Big| [\varphi(x)-\varphi(a)]^{\alpha} u(x) \Big|.
	\end{equation}
	Moreover, the weighted space $C_{\alpha,\varphi}^{m} [a, b]$ of functions $ u $ on $ (a,b] $ is defined by
	\begin{equation*} \label{generalizedweightedspace2}
		C_{\alpha,\varphi}^{m} [a, b] = \{ u : (a,b] \to \R\;\; :\;\; u(x) \in C^{m-1} [a,b],\; u^{(m)}(x) \in C_{\alpha,\varphi} [a,b]  \},
	\end{equation*}
	with the norm 
	\begin{equation*}
		\|u\|_{C_{\alpha,\varphi}^{m}} = \sum_{k=0}^{m-1} \left\| u^{(k)} \right\|_{\infty} + 	\left\|u^{(m)}\right\|_{C_{\alpha,\varphi}},
	\end{equation*}
	where $\|\cdot\|_{\infty}$ is the usual supremum norm on $C[a,b]$.
\end{definition}

\begin{theorem}\label{thm4.2}
	With all notation as in Definition \ref{Def:GD} and \ref{weightedspaces}, and with $\overline{\mathcal{A}}$ assumed to be analytic on the disc $\mathcal{D}(0,R)$ mentioned in Definition \ref{def:F}, we have a well-defined bounded operator $\prescript{\mathcal{A}_R}{a}D^{\alpha,\beta}_{\varphi(x)}$ on the space $C_{\alpha,\varphi}^{} [a, b]$ and
	\begin{equation}
		\left\|  \prescript{\mathcal{A}_R}{a}D^{\alpha,\beta}_{\varphi(x)} u \right\|_{C_{\alpha,\varphi}} \leq M \|u\|_{C_{\alpha,\varphi}^{m}},
	\end{equation}
	where the constant $M$ is defined by
	\begin{equation} 
		M = \frac{\left(\varphi(b)-\varphi(a)\right)^{m-\alpha}}{m-\alpha} \sup_{\left| \tau \right|< \left(  \varphi(b)-\varphi(a) \right)^{\beta} } \left| \mathcal{\overline{A}}(\tau) \right|.
	\end{equation}
\end{theorem}
\begin{proof}
	Using the definition of $  \prescript{\mathcal{A}_R}{a}D^{\alpha,\beta}_{\varphi(x)} $ and Eq. \eqref{Calphaphinorm}, we obtain
	\begin{align*}
		\left\|  \prescript{\mathcal{A}_R}{a}D^{\alpha,\beta}_{\varphi(x)} u \right\|_{C_{\alpha,\varphi}} &= \left\| \left(\frac{1}{\varphi'(x)}\cdot\frac{\mathrm{d}}{\mathrm{d}x}\right)^m\prescript{\mathcal{\overline{A}}}{a}I^{m-\alpha,\beta}_{\varphi(x)} u \right\|_{C_{\alpha,\varphi}} \\
		&= \max_{x \in [a,b]} \left| [\varphi(x)-\varphi(a)]^{\alpha} \left(\frac{1}{\varphi'(x)}\cdot\frac{\mathrm{d}}{\mathrm{d}x}\right)^m\prescript{\mathcal{\overline{A}}}{a}I^{m-\alpha,\beta}_{\varphi(x)}u(x) \right|\\
		&= \max_{x \in [a,b]} \bigg| [\varphi(x)-\varphi(a)]^{\alpha} \left(\frac{1}{\varphi'(x)}\cdot\frac{\mathrm{d}}{\mathrm{d}x}\right)^m \\&\hspace{3cm}\times\int_{a}^{x} \varphi'(t)[\varphi(x)-\varphi(t)]^{m-\alpha-1} \mathcal{\overline{A}}[(\varphi(x)-\varphi(t))^{\beta}]u(t)\,\mathrm{d}t  \bigg| \\
		&\leq \|u\|_{C_{\alpha,\varphi}^{m}}\max_{x \in [a,b]} \left| \int_{a}^{x} \varphi'(t)[\varphi(x)-\varphi(t)]^{m-\alpha-1} \mathcal{\overline{A}}[(\varphi(x)-\varphi(t))^{\beta}]\,\mathrm{d}t  \right|.
	\end{align*}
	Making the substitution $\varphi(\tau) =\varphi(x)-\varphi(t) $ in the above integral, which gives $\varphi'(\tau)\,\mathrm{d}\tau =-\varphi'(t)\,\mathrm{d}t$, we get	
	\begin{equation*}
		\left\|  \prescript{\mathcal{A}_R}{a}D^{\alpha,\beta}_{\varphi(x)} u \right\|_{C_{\alpha,\varphi}} \leq \|u\|_{C_{\alpha,\varphi}^{m}}\max_{x \in [a,b]} \Big|   \int_{0}^{\varphi^{-1} \left( \varphi(x)-\varphi(a) \right)}  \varphi'(\tau)\varphi(\tau)^{m-\alpha-1} \mathcal{\overline{A}}[\varphi(\tau)^{\beta}]\,\mathrm{d}\tau  \Big|.
	\end{equation*}
	Because $ \mathcal{\overline{A}} $ is analytic on the disc $\mathcal{D}(0,R)$, the function $  \mathcal{\overline{A}}[\varphi(\tau)^{\beta}] $ is bounded on the finite interval $ [0,\varphi^{-1} \left( \varphi(x)-\varphi(a) \right)] $. Thus, we get
	\begin{equation*}
		\left\|  \prescript{\mathcal{A}_R}{a}D^{\alpha,\beta}_{\varphi(x)} u \right\|_{C_{\alpha,\varphi}}  \leq M \|u\|_{C_{\alpha,\varphi}^{m}},
	\end{equation*}
	where
	\begin{align*} 
		M &:= \max_{x \in [a,b]}\sup_{\left| \tau \right|< \left(  \varphi(x)-\varphi(a) \right)^{\beta} } \left| \mathcal{\overline{A}}(\tau) \right|  \Big|   \int_{0}^{\varphi^{-1} \left( \varphi(x)-\varphi(a) \right)}  \varphi'(\tau)\varphi(\tau)^{m-\alpha-1} \,\mathrm{d}\tau   \Big|\\&=\frac{\left(\varphi(b)-\varphi(a)\right)^{m-\alpha}}{m-\alpha} \sup_{\left| \tau \right|< \left(  \varphi(b)-\varphi(a) \right)^{\beta} } \left| \mathcal{\overline{A}}(\tau) \right|,
	\end{align*}
	and this completes the proof.
\end{proof}

\section{Generalised Laplace transform and fractional differential equations} \label{chap:Lap}

Transform methods, including Laplace, Fourier, and other related transforms, are a tried and tested way of solving differential equations, both ordinary and partial, and also fractional ones. In the setting of fractional operators with respect to a function $\varphi$, we need a generalisation of the Laplace transform which is appropriate for handling those operators.

In this section, we first recall the definitions of generalised Laplace transform and $ \varphi $-exponential order which we shall refer to throughout the section. We find the generalised Laplace transforms of the generalised fractional integral and derivative of a function with respect to another function, and then make use of these results to solve some differintegral equations in the settings of the generalised operators.

\begin{definition}[\cite{jarad2020generalized,fahad2020laplace}] \label{psilaplacedef}
	Let $ u:[0,\infty) \to \mathbb{R} $ be a real-valued function and $ \varphi $ be a non-negative increasing $C^1$ function such that $\varphi(0)=0$. Then the Laplace transform of $u$ with respect to $\varphi$ is defined by
	\begin{equation} \label{def}
		\mathcal{L}_{\varphi} \left\{ u(x) \right\}=\hat{u}(s)=\int_{0}^{\infty} e^{-s \varphi(x)} \varphi^\prime (x) u(x) \,\mathrm{d}x,
	\end{equation}
for all $s\in\mathbb{C}$ such that this integral converges. Here $\mathcal{L}_{\varphi}$ denotes the Laplace transform with respect to $\varphi$, which we call a generalised Laplace transform.
\end{definition}

\begin{definition}[\cite{jarad2020generalized,fahad2020laplace}] \label{Def:psiexp}
	A function $u:[0,\infty)\to\mathbb{R}$ is said to be of $ \varphi $-exponential order $c>0$ if there exist positive constants $ M $ and $X$ such that for all $ x>X $,
	\[
	\left|u(x)\right| \leq M e^{c \varphi(x)},
	\]
	i.e., if
	\[
	u(x)= \mathcal{O}( e^{c \varphi(x)})  \text{ \ \ \ \ as \ \   }  x \to \infty.
	\]
	The significance of this definition is that, if $u$ is of $\varphi$-exponential order $c$, then its Laplace transform with respect to $\varphi$, the function $\mathcal{L}_{\varphi} \left\{ u(x) \right\}=\hat{u}(s)$, is well-defined for all $s\in\mathbb{C}$ with $\Real(s)>c$.	
\end{definition}

\begin{theorem}\label{GL:IO}
	With all notation as in Definition \ref{Def:GF}, and assuming that the function $ u $ is of $ \varphi $-exponential order, piecewise continuous over each finite interval $ [0,T] $, with generalised Laplace transform $ \hat{u} $, the function $	\prescript{\mathcal{A}}{0}I^{\alpha, \beta}_{\varphi(x)}u(x)$ has a generalised Laplace transform given by the following formula: 
	\begin{equation}
		\widehat{\prescript{\mathcal{A}}{0}I^{\alpha, \beta}_{\varphi(x)}u}(s)= s^{-\alpha} \mathcal{A}_{\Gamma}(s^{-\beta}) \hat{u}(s).
	\end{equation}
\end{theorem}
\begin{proof}
	Using the series formula \eqref{53} along with the generalised Laplace transform of the fractional integral of a function with respect to another function \cite{jarad2020generalized,fahad2020laplace}, we get
	\begin{align*}
		\widehat{\prescript{\mathcal{A}}{0}I^{\alpha, \beta}_{\varphi(x)}u}(s)&=\sum_{n=0}^{\infty} a_n \Gamma(\beta n+\alpha)\widehat{\prescript{RL}{0}I^{\beta n+ \alpha}_{\varphi(x)}} u(s)\\&=\sum_{n=0}^{\infty} a_n \Gamma(\beta n+\alpha) s^{-\beta n-\alpha} \hat{u}(s)\\&=s^{-\alpha} \hat{u}(s)\sum_{n=0}^{\infty} a_n \Gamma(\beta n+\alpha) s^{-\beta n},
	\end{align*}
	as required.
\end{proof}

\begin{theorem} \label{Thm:GLTderiv}
	With all notation as in Definition \ref{Def:GD}, and assuming that the function $ u $ is such that $ u(x) $, $ \prescript{RL}{0}I^{m - \alpha}_{\varphi(x)} u(x)$, $\prescript{RL}{0}D^{1}_{\varphi(x)} \prescript{RL}{0}I^{m - \alpha}_{\varphi(x)} u(x)$, $\dots$, $\prescript{RL}{0}D^{m}_{\varphi(x)} \prescript{RL}{0}I^{m - \alpha}_{\varphi(x)} u(x)$ are continuous on $ (0,\infty) $ and of $ \varphi $-exponential order, while $ \prescript{RL}{0}D^{\alpha}_{\varphi(x)}u(x) $ is piecewise continuous on $ [0,\infty) $, the function $\prescript{\mathcal{A}}{a}D^{\alpha,\beta}_{\varphi(x)}u(x)$ has a generalised Laplace transform given by the following formula: 
	\begin{align*}
		\widehat{\prescript{\mathcal{A}}{a}D^{\alpha,\beta}_{\varphi(x)}u}(s)= \sum_{n=0}^{\infty} \overline{a}_n &\Gamma(\beta n-\alpha+m) s^{\alpha-\beta n} \hat{u}(s) \\&-\sum_{n=0}^{\lfloor\Real\alpha/\Real\beta\rfloor}  \sum_{i=0}^{\lfloor N_n \rfloor} \overline{a}_n \Gamma(\beta n-\alpha+m)s^{N_n-i} \left( \prescript{RL}{0}I^{N_n-i + \beta n - \alpha +1 }_{\varphi(x) } u \right)(0),
	\end{align*}
	where $N_n\in\mathbb{Z}$ is defined as $N_n=\lfloor\Real(\alpha-\beta n)\rfloor$ for all $n$ such that this quantity is non-negative.
\end{theorem}
\begin{proof}
Using the series formula \eqref{59} along with the generalised Laplace transform of the fractional derivative of a function with respect to another function, we obtain
	\[
	\widehat{\prescript{\mathcal{A}}{a}D^{\alpha,\beta}_{\varphi(x)}u}(s)=\sum_{n=0}^{\infty} \overline{a}_n \Gamma(\beta n-\alpha+m)\widehat{\prescript{RL}{a}D^{\alpha-\beta n}_{\varphi(x)}u}(s),
	\]
		where
	\[\widehat{\prescript{RL}{a}D^{\gamma}_{\varphi(x)}u}(s)=s^{\gamma} \hat{u}(s) - \sum_{i=0}^{\lfloor \Real\gamma \rfloor} s^{\lfloor\Real\gamma\rfloor-i} \left( \prescript{RL}{0}I^{\lfloor\Real\gamma\rfloor-i -\gamma +1 }_{\varphi(x) } u \right)(0),\qquad\Real(\gamma)>0.
	\]
This gives the required result, noting that there are only finitely many values of $n\geq0$ such that $\Real(\alpha-\beta n)\geq0$, namely all the values of $n\leq\frac{\Real(\alpha)}{\Real(\beta)}$, so the eventual resulting sum is finite.
\end{proof}

\begin{remark}
The assumptions on $u$ in the above Theorem are not all independent from each other. Indeed, if a function $u$ is $\varphi$-exponentially bounded, then so is its fractional integral to any positive (or positive-real-part) order. Therefore, in Theorem \ref{Thm:GLTderiv}, it is sufficient to assume that $u(x)$, and also $\prescript{RL}{0}D^{k}_{\varphi(x)} \prescript{RL}{0}I^{m - \alpha}_{\varphi(x)} u(x)$ for sufficiently large $k\leq m$ so that this is a fractional derivative rather than a fractional integral, are exponentially bounded.
\end{remark}

\begin{theorem}
	With all notation as in Definition \ref{Def:GF}, and assuming that the function $ v $ is piecewise continuous and is of $ \varphi $-exponential order, the following fractional integral has a unique solution $ u $.
	\begin{equation}\label{Solinteq}
		\prescript{\mathcal{A}}{0+}I^{\alpha, \beta}_{\varphi(x)}u(x) + c u(x)= v(x), \ \ \ u(0)=\frac{v(0)}{c}, \ \ c \in \R.
	\end{equation}
\end{theorem}
\begin{proof}
	Applying the generalised Laplace transform to the equation \eqref{Solinteq} and using the result of Theorem \ref{GL:IO}, we get
	\begin{align*}
		\hat{v}(s) &= \widehat{\prescript{\mathcal{A}}{0+}I^{\alpha, \beta}_{\varphi(x)}u}(s) + c \hat{u}(s)\\&= s^{-\alpha} \mathcal{A}_{\Gamma}(s^{-\beta}) \hat{u}(s) + c \hat{u}(s). 
	\end{align*}
	So we have an explicit expression for the generalised Laplace transform of $ u $, namely
	\begin{equation*}
		\hat{u}(s)= \frac{\hat{v}(s)}{s^{-\alpha} \mathcal{A}_{\Gamma}(s^{-\beta})  + c }.
	\end{equation*}
	This has a unique inverse generalised Laplace transform, giving a unique solution function $ u $.
\end{proof}

\section{Applications and conclusions} \label{chap:concl}

The main goal of this paper was to study the general class of fractional operators with analytic kernels with respect to functions, in particular using convergent infinite series and an operational calculus formulation. We achieved the main purpose, establishing various properties of these operators, such as establishing appropriate composition properties and function spaces, by considering each problem using the same methods as used in one of the smaller pre-existing classes of fractional operators. To achieve these goals, several concepts from fractional calculus and its generalisations needed to be pieced together. Key roles were played by the concept of series formulae, which comes from the theory of fractional calculus with analytic kernels, and by the concept of conjugation of operators, which comes from the theory of fractional calculus with respect to functions.

As we draw to the end of this work, it is time to consider not only the achievements made so far, but also finding open doors to several directions of further research that could be pursued in both the short term and long term. Mathematically, more work could be done on the function spaces appropriate for these operators, and on their relationships with other fractional operators. The current work contains some basic results on fractional differential equations using the new generalised operators, but these could be extended in the future to more advanced techniques for fractional differential equations which might find some value in applications.

As always, general mathematical formulations are valuable largely because they can be applied to a broader range of different real-world problems. It is known that several operators defined using convolutions with analytic kernels have discovered applications in effective modelling, such as the Atangana--Baleanu operators displaying Mittag-Leffler behaviour \cite{bas-acay,jajarmi}. Fractional derivatives with respect to functions have also discovered applications, such as the Hadamard operators displaying logarithmic behaviour which appear in probabilistic modelling of certain statistical distributions \cite{garra-orsingher-polito}, and fractional derivatives with respect to exponential functions used to effectively extend the Dodson diffusion equation \cite{garra-giusti-mainardi}. The general formalism introduced in this paper covers all such operators under a single umbrella, and also others such as Hadamard-type operators (see Example \ref{Ex:HT} above) which were not covered by either of the previous classes, but which nevertheless have applications e.g. in kinetics \cite{ma}. Therefore, any results proved in this general setting can immediately be applied to a huge range of problems, anything involving any operator with an analytic kernel or with respect to a monotonic function or any combination of the above.

The main point of this work is to demonstrate the ultimate generality to which such operators of fractional calculus can be taken: combining two already very general classes of operators to obtain a new, even more general, class. Even if this class in its full generality is not applicable to solve many real-world problems directly, its advantage lies in the fact that it is broad and general enough to cover a huge number of types of fractional calculus, from Prabhakar to Erdelyi to Hadamard-type, which all have different properties and behaviours, but which can all be covered by the general formulae introduced here.

\bibliography{OFDA}
\bibliographystyle{unsrt}
\end{document}